\documentclass[a4paper]{article}
\usepackage[T1]{fontenc}
\usepackage[utf8]{inputenc}

\usepackage{fancyhdr}

\usepackage{amsmath}
\usepackage{amssymb}
\usepackage{amsfonts}
\usepackage{mathrsfs}
\usepackage{amsthm}

\usepackage{graphicx}
\usepackage{newlfont}
\usepackage{latexsym}

\theoremstyle{plain}
\newtheorem{teo}{Theorem}[section]
\newtheorem{cor}[teo]{Corollary}
\newtheorem{lem}[teo]{Lemma}
\newtheorem{conj}[teo]{Conjecture}
\newtheorem{prop}[teo]{Proposition}

\theoremstyle{definition}
\newtheorem{dfn}[teo]{Definition}

\newtheorem{oss}[teo]{Remark}

\DeclareMathOperator{\rk}{rk}

\title{On matroid modularity and the coefficients of the inverse Kazhdan-Lusztig polynomial of a matroid }
\author{Lorenzo Vecchi}
\date{\today}

\begin{document}

\maketitle
\begin{abstract}
    Following the work of Gao and Xie in \cite{invkl}, we state some properties of the inverse Kazhdan-Lusztig polynomial of a matroid. We also give partial answers to a conjecture that states that regular connected matroids are non-degenerate. We link the degeneracy of a matroid to the inverse Kazhdan-Lusztig polynomial and we show that the Conjecture holds for modular matroids, by proving that degenerate modular matroids are not regular.
\end{abstract}
\section{Introduction}
Kazhdan-Lusztig polynomials for matroids were firstly introduced in 2016 by Elias, Proudfoot and Wakefield in \cite{kl}, in analogy with the classical Kazhdan-Lusztig polynomials for Coxeter groups (\cite{classickl}). While the classical version gives rise to a very broad class of polynomials (every polynomial with non-negative integer coefficients and constant term equal to 1 is the Kazhdan-Lusztig polynomial of some Coxeter group \cite{pol}), the matroidal Kazhdan-Lusztig polynomials have a very rigid structure: Braden et al. proved in \cite{singhodgetheory} that their coefficients are always non-negative, and it is conjectured that they form a log-concave sequence with no internal zeros. In 2020, Gao and Xie introduced in \cite{invkl} the analogue of the inverse Kazhdan-Lusztig polynomial for a matroid, employing the general framework of Kazhdan-Lusztig-Stanley theory. These polynomials were also proven to have non-negative coefficients in \cite{singhodgetheory}, and they are conjectured as well to be log-concave with no internal zeros.
The properties of these polynomials, which will be called $P_M(t)$ and $Q_M(t)$ in accordance to all the present literature, are usually not understandable at first glance, since their definitions are given through a quite intricate recursion.
Another Conjecture, found in \cite{survey}, gives a statement about the degree of these polynomials. We say that a matroid $M$ of rank $r$ is degenerate if the degree of $P_M(t)$ is strictly less than $\left\lfloor \frac{r-1}{2} \right\rfloor$. Then one has the following
\begin{conj}\label{nondegconj}
Every connected regular matroid is non-degenerate.
\end{conj}
Our main aim is to develop new results for $P_M(t)$ and $Q_M(t)$ in a purely combinatorial way, to give partial answers to Conjecture \ref{nondegconj}.
In Section \ref{coefficients},
we give an explicit definition of the coefficient of degree 1 of $Q_M(t)$ (Theorem \ref{tQ}). 
We then move on to consider the class of modular matroids, which are proven to be the only matroids with a zero-degree Kazhdan-Lusztig polynomial. In a sense, in this setting, modularity is equivalent to maximal degeneracy, since, even for high-rank matroids, $P_M(t)$ remains constant.
In Section \ref{conjecture} we give some results on how non-degeneracy is linked to $Q_M(t)$ and observe how even and odd rank matroids behave quite differently. Lastly, Theorem \ref{modregular} settles the Conjecture for this class of matroids, by showing that rank-3 or higher modular matroids are either not connected or not regular.
\\

\textbf{Acknowledgements.} The author is very grateful to June Huh, for kindly pointing out an obvious mistake in a previous version of the article, and to Luca Moci, Michele D'Adderio, Simone Billi and Aleksandr Dzhugan for the thoughtful discussions and the helpful comments.

\section{Background and notation}
We begin our work by setting all the notation needed in the rest of the paper. Many invariants defined on matroids can actually be calculated employing an order-theoretic approach by working on their family of flats, which form a lattice by inclusion. To start, we recall the definition of a geometric lattice.
\begin{dfn}
A lattice $\mathscr{L}$ is said to be atomic if each non-minimal element can be obtained as a join of its atoms. It is semimodular if it has a rank function $\rho$ such that
$$\rho(F)+\rho(G) \geq \rho(F\vee G) + \rho(F\wedge G),$$
for all $F,G \in \mathscr{L}$. If $\mathscr{L}$ is atomic and semimodular it is said to be geometric.
\end{dfn}
We are also interested in the underlying matroidal structure of $\mathscr{L}$. It is a theorem by Birkhoff \cite{birk} that every geometric lattice arises as a lattice of flats of a matroid. If $M$ is a matroid with lattice of flats isomorphic to $\mathscr{L}$ we will write $\mathscr{L}=\mathscr{L}(M)$, omitting the $M$ if the matroid is clear from the context. In this sense, we will also say that $\mathscr{L}$ is a matroid.\\
Unless otherwise stated, we will always work with simple matroids of rank $r$. Given a simple matroid $M$ we denote with $E$ its ground set (which is equal to the set of atoms of $\mathscr{L}$), and $\chi_M(t)$ its characteristic polynomial. Moreover, we call $[\emptyset, F]\subset \mathscr{L}$ the restriction of $M$ to $F$, denoted with $M_F$, and we call $[F,E]\subset \mathscr{L}$ the contraction of $M$ by $F$ denoted by $M^F$. It is important to observe that simplicity is easily lost when performing contractions and needs to be handled carefully in recursions; if a proposition needs to be stated differently in the non-simple case we will do so; otherwise we can always suppose to work with a simplified version of matroids (removing all loops and parallel elements).\\
The uniform matroid of rank $k$ on a ground set $|E|=n$ will be denoted with $U_{k,n}$. The Boolean matroid, corresponding to the Boolean lattice, will be denoted with $B_n$.\\
We denote with $w_k$ and $W_k$, respectively, the Whitney numbers of first and second kind. We recall that $\chi_M(t)=\sum_{k=0}^r w_kt^{r-k}$ and that $W_k$ is the number of rank-k flats in the matroid.
See \cite{oxley} and \cite{white} for exact definitions on matroids and geometric lattices.\\

The notion of modularity in a lattice can be stated in several different ways, but we collect in the following proposition the different characterizations we will need for our results.
\begin{prop}
For a lattice $\mathscr{L}$ the following statements are equivalent
\begin{itemize}
    \item[i)] $\mathscr{L}$ is modular,
    \item[ii)] $\mathscr{L}$ has a rank function $\rho$ such that $\rho(F)+\rho(G) = \rho(F\vee G) +\rho(F\wedge G)$ for any two elements $F$ and $G$,
    \item[iii)] (Diamond isomorphism Theorem) For every $F$ and $G$ in $\mathscr{L}$ the intervals
    $$[F\wedge G, F]\cong [G, F\vee G]$$
    are isomorphic.
\end{itemize}
If, moreover, $\mathscr{L}$ is geometric, then $i)$ is also equivalent to
\begin{itemize}
    \item[iv)] $W_1 = W_{r-1}$,
    \item[v)] $\mathscr{L}$ is the product of a boolean matroid and a projective geometry,
    \item[vi)] The dual lattice $\mathscr{L}^* = (\mathscr{L},\geq)$ is geometric.
\end{itemize}
\end{prop}
Condition $iv)$ is the Hyperplane Theorem and can be found for example in \cite{white}, Chapter 8. Condition $v)$ is a Theorem by Birkhoff (\cite{birk2}); we recall that a projective geometry is a geometric lattice where
\begin{itemize}
    \item Every rank-2 flat contains at least three atoms,
    \item Every hyperplane intersect every rank-2 flat.
\end{itemize}
Regarding condition $vi)$, geometric lattices are always dually atomic (\cite{stern}), hence what one has left to prove is the semimodularity of $\mathscr{L}^*$. Moreover, the Top-Heavy Conjecture, proved in \cite{singhodgetheory}, implies that
\begin{cor}
If $\mathscr{L}$ is geometric and modular, its sequence of Whitney numbers of the second kind is symmetric, i.e. for all $k$
$$W_k = W_{r-k}.$$
\end{cor}

Theorem 2.2 in \cite{kl} and Theorem 1.2 in \cite{invkl} give us the definition of the Kazhdan-Lusztig polynomial and the inverse Kazhdan-Lusztig polynomial of a matroid, which we will now recall.
\begin{teo}\label{PQ}
There is a unique way to assign to each matroid $M$ two polynomials $P_M(t), Q_M(t) \in \mathbb{Z}[t]$ such that
\begin{itemize}
    \item[i)] If $r = 0$, then $P_M(t) = 1$.
    \item[ii)] If $r > 0$, then $\deg P_M(t) <\frac{1}{2}r$.
    \item[iii)] For every matroid $M$ we have
    $$t^r P_M(t^{-1}) = \sum_{F \in \mathscr{L}(M)}\chi_{M_F}(t)P_{M^F}(t).$$
    \item[i')] If $r = 0$, then $Q_M(t) = 1$.
    \item[ii')] If $r > 0$, then $\deg Q_M(t) <\frac{1}{2}r$.
    \item[iii')] For every matroid $M$ we have
    $$t^r \cdot (-1)^r Q_M(t^{-1}) = \sum_{F \in \mathscr{L}(M)}(-1)^{\rk M_F} Q_{M_F}(t)\cdot t^{\rk M^F}\chi_{M^F}(t^{-1}).$$
\end{itemize}
\end{teo}

\begin{oss}\label{conv}
These polynomials also arise in the more general framework of Kazhdan-Lusztig-Stanley theory. We call $f$ the
unique function that satisfies
$$\bar{f} = \chi * f,$$
where $\bar{f}$ is the involution $\bar{f}(t) = t^r f(t^{-1})$ and $*$ denotes the convolution in the incidence algebra of $\mathscr{L}$. Then $f$ is called the right Kazhdan-Lusztig-Stanley function associated to $\chi$. In this setting then one has that
$$
P_M(t) = \left(f\right)_{\emptyset,E}(t).
$$
Then, the polynomial $Q_M(t)$ can be found to be equal to
$$Q_M(t) = (-1)^r\left(f^{-1}\right)_{\emptyset, E}(t),$$
where $f^{-1}$ is the inverse of $f$ in the incidence algebra of $\mathscr{L}$. This justifies the name "inverse polynomial" for $Q_M(t)$, since it arises as the inverse of $P_M(t)$ with respect to the convolution product. In particular, this means that
$$\sum_{F \in \mathscr{L}(M)} P_{M_F}(-1)^{\rk M^F}Q_{M^F}=\sum_{F \in \mathscr{L}(M)} P_{M^F}(-1)^{\rk M_F}Q_{M_F}=0.$$
\end{oss}

\begin{oss}\label{q}
To ease some of the computations, we will also use this version of condition $iii')$,
$$
Q_M(t) = (-1)^r \sum_{F \in \mathscr{L}(M)}(-1)^{\rk F}t^{\rk F} Q_{M_F}(t^{-1})\cdot \chi_{M^F}(t).
$$
\end{oss}

\section{New results on the coefficients of $P_M(t)$ and $Q_M(t)$}\label{coefficients}

We will denote the coefficient of $t^j$ in a polynomial $p(t)$ as $[t^j]p(t)$.

It is observed in Proposition 2.4 in \cite{invkl} that
\begin{prop}\label{t0}
$$[t^0]Q_M(t) = |[t^0]\chi_M(t)|=|w_r|=|\mu(M)|.$$
\end{prop}

Our next goal is to compute the coefficient of $t$ in the same fashion of Proposition 2.12 in \cite{kl}. To do that we recall the definition of the doubly-indexed Whitney numbers.
\begin{dfn}
For all natural numbers $i$ and $j$, let
$$
w_{i,j} := \sum_{\rk F=i, \rk G=j}\mu(F,G) \hspace{0.5 cm} \text{and} \hspace{0.5 cm} W_{i,j}:=\sum_{\rk F = i, \rk G = j}\zeta(F,G),
$$
where $\mu$  is the Moebius function of $\mathscr{L}(M)$ and $\zeta$ is its inverse in the incidence algebra of $\mathscr{L}(M)$.
If $i$ is set to 0, we recover the Whitney numbers of first and second kind.
\end{dfn}

\begin{teo}\label{tQ}
For every matroid $M$,
$$[t]Q_M(t)=|w_{1,r}|-|w_{0,r-1} |.$$
\end{teo}
\begin{proof}
We just need to compute the coefficient of $t$ in the right-hand side of the equation in Remark \ref{q}. First of all we notice that flats with rank greater than or equal to 2 will give no contribution to the coefficient of $t$ because the terms of $Q_{M_F}(t^{-1})$ have degrees ranging in $(-\frac{\rk F}{2}, 0]$,
which means that the lowest term of $$t^{\rk F}Q_{M_F}(t^{-1})\chi_{M^F}(t)$$ has degree strictly greater than 1. If we consider $F=\emptyset$, then the corresponding term in the sum is reduced to $(-1)^{\rk M} \chi_M(t)$,
since $Q_{M_\emptyset}(t) = 1$.
Then we get a contribution for the coefficient of $t$ of $(-1)^{\rk M}[t]\chi_M(t)$. The sign of $[t]\chi_M(t)$ is opposite to the parity of $\rk M$, which means that
$$(-1)^{\rk M}[t]\chi_M(t) = -|[t]\chi_M(t)|.$$

Let us now consider a rank-1 flat, $i\subset E$. Combining $ii')$ and \ref{t0} we get that $Q_{M_i}(t) = Q_{B_1}(t) \equiv 1$, thus its corresponding term in the sum is 
$$(-1)^{\rk M -1} t \chi_{M/i}(t).$$
This means that each of these terms gives a contribution to the coefficient of $t$ of $(-1)^{\rk M -1}[t^0]\chi_{M/i}(t)$, which, for the same reason as before, is equal to $|[t^0]\chi_{M/i}(t)|$. Putting everything together tells us that
$$[t]Q_M(t) = \left|\sum_{i\in E} [t^0]\chi_{M/i}(t) \right|- \left|[t]\chi_M(t)\right|.$$
A straightforward substitution using the definition of $\chi_M(t)$ gives us the desired result.
\end{proof}

We now continue by investigating the behaviour of modular matroids with respect to their Kazhdan-Lusztig polynomials. Proposition 2.14 in \cite{kl} gives the following condition on $P_M(t)$ regarding the modularity of $\mathscr{L}(M)$.

\begin{prop}
The following conditions are equivalent 
\begin{itemize}
    \item[i)] $P_M(t)=1$,
    \item[ii)] $[t]P_M(t) = 0$,
    \item[iii)] $\mathscr{L}(M)$ is modular.
\end{itemize}
\end{prop}

In a sense then, modular matroid are the most degenerate ones and it should be clear by now why we are interested in studying them with regards to Conjecture \ref{nondegconj}. We can now give necessary and sufficient conditions for modularity with respect to $Q_M(t)$.

\begin{teo}
The modularity of $\mathscr{L}(M)$ is equivalent to $\deg Q_M(t) = 0$.
\end{teo}
\begin{proof}
The proof is by induction on $r$. From the relation
$$(-1)^rQ_M(t) = -\sum_{F\neq\emptyset }P_{M_F}(t)(-1)^{r-\rk F}Q_{M^F}(t),$$
if $\mathscr{L}(M)$ is modular, then so is the lattice of any minor of $M$. By induction then, $(-1)^rQ_M(t)$ is constant. Conversely, suppose that $M$ is not modular but $\deg Q_M(t)=0$, and that $M$ is minimal by contraction, that is $M$ is not modular and $M^F$ is modular for every $F\neq \emptyset$. Then,
$$P_M(t) =-\sum_{F\neq \emptyset}(-1)^{\rk F}Q_{M_F}(t).$$
If every $M_F$ is modular, then $P_M(t)$ is constant, which is a contradiction. Consider, then, a flat $F$ such that $M_F$ is not modular and that is minimal by restriction, that is $M_G$ is modular for every $G\subseteq F$. Since by induction hypothesis $M^G$ is modular for every flat $\emptyset \neq G\subseteq E$, then $M^G_F$ is modular as well, proving that $M_F$ is indeed modular. 
\end{proof}

\section{New results on the Conjecture of degeneracy}\label{conjecture}
We now finally move on to give some partial answers to Conjecture \ref{nondegconj}.

Let us first study how the polynomial $Q_M(t)$ behaves with respect to non-degeneracy.
\begin{teo}\label{all}
Let $M$ be a matroid of odd rank $r$. Then $$\left[t^{\left\lfloor \frac{r-1}{2}\right\rfloor}\right]P_M(t) = \left[t^{\left\lfloor \frac{r-1}{2}\right\rfloor}\right]Q_M(t).$$
In particular, a matroid $M$ of odd rank $r$ is non-degenerate if and only if $Q_M(t)$ has degree $\left\lfloor\frac{r-1}{2}\right\rfloor$.   
\end{teo}
\begin{proof}
Let $r=2k+1$; this becomes a special case of Corollary 8.8 in \cite{stan}.
We know that
$$\sum_F P_{M_F}(t)(-1)^{\rk M^F}Q_{M^F}(t) =0.$$
If $F=\emptyset$ or $F=E$ we get, respectively, $P_M(t)$ and $(-1)^{2k+1}Q_M(t)=-Q_M(t)$. Let us show that no other term in the sum can have degree $k$. If $\emptyset \neq F \neq E$, 
\begin{align*}
  \deg P_{M_F}Q_{M^F}=&\deg P_{M_F} + \deg Q_{M^F}\leq \left\lfloor \frac{\rk F -1}{2}\right\rfloor + \left\lfloor \frac{r-\rk F -1}{2}\right\rfloor \\
  <& \frac{r}{2}-1=k-\frac{1}{2}.  
\end{align*}
Hence, 
$$
0 = [t^k]0 = [t^k]\left(P_M(t)-Q_M(t)\right) = [t^k]P_M(t) - [t^k]Q_M(t).
$$
\end{proof}

The Theorem is false in general for even rank. If, for example, $r=2$, because the maximal degree is $0$ but Proposition \ref{t0} gives infinite counterexamples since $[t^0]\chi_{U_{2,n}}(t) = n-1$. More specifically, the only rank-2 matroid for which the statement holds is the Boolean matroid $B_2$.

Moreover, there exist both examples where the coefficient $\left[ t^{k-1} \right]Q_M(t)$ of a rank-($2k$) matroid is either greater or smaller than $\left[t^{k-1} \right]P_M(t)$. For example,
$$
[t^2]Q_{U_{6,7}}=14 < 21 =[t^2]P_{U_{6,7}}.
$$

Unfortunately, a statement on the degree of $Q_M(t)$ is much more difficult to get. Let us start with a preliminary result.
\begin{lem}\label{Mi}
If $M$ is simple and connected, then there exists an element $i\in E$ such that $\mathscr{L}(M^i)$ is isomorphic to the lattice of flats of a simple connected matroid.
\end{lem}
\begin{proof}
The proof is by induction on $n$. If $n=r$, $M\cong B_n$ is not connected, thus we can assume $n\geq r+1$. If $n=r+1$, then $M\setminus i$ is a matroid on $r$ elements of rank $r$ (since $M$ is connected it does not contain coloops and the deletion of one element does not decrease the rank). This implies that $M^i$ is connected for every $i\in E$ (\cite{tutte} Theorem 6.5). If now $n>r+1$, and every $M^i$ is not connected, this implies that $M\setminus i$
is connected and simple for every $i\in E$. By induction hypothesis, for every $j\neq i$, $\left(M\setminus i\right)^j = \left(M^j\right)\setminus i$ is connected. However, since $M^j$ is not connected, this means that $M^j$ has two connected components, namely, $i$ and $E\setminus(i\cup j)$, and in particular that $i$ is a coloop in $M^j$ (it cannot be a loop because that would imply that $i$ and $j$ are parallel). This means that $i$ is not contained in any circuit of $M^j$, or, in other words, that no circuit of $M$ contains both $i$ and $j$, thus contradicting the hypothesis on $M$ being connected.  
\end{proof}

Now, if $M$ has even rank $r = 2k$, then
$$
[t^{k-1}]\left(P_M(t) + Q_M(t) \right) = \sum_{\substack{F\in\mathscr{L}(M)\\\rk F \text{ odd}}}[t^{k-1}]P_{M^F}Q_{M_F}.
$$
\begin{prop}
Suppose Conjecture \ref{nondegconj} is true in odd rank. If the coefficient $[t^{k-1}]Q_M(t)$ is zero for every degenerate matroid $M$ of even rank, then the Conjecture holds on even rank matroids.  
\end{prop}
\begin{proof}
Suppose, by contradiction, that $M$ is connected, regular and degenerate. Then we know that
$$
Q_M(t) = \sum_{\substack{F\in\mathscr{L}(M)\\\rk F \text{ odd}}}[t^{\left\lfloor \frac{2k-\rk F-1}{2}\right\rfloor}]P_{M^F}(t)[t^{\left\lfloor\frac{\rk F -1}{2} \right\rfloor}]Q_{M_F}
$$
and by Lemma \ref{Mi} there exists $i\in E$ such that $M^i$ is connected and regular; since $\rk M^i$ is odd, we can assume it is non degenerate, that is $[t^{k-1}]P_{M^i}>0$. This means that $[t^{k-1}]Q_M(t) > 0$.
\end{proof}

Finally, we are able to prove that the Conjecture holds for all modular matroids.
We recall that the Fano matroid is a modular matroid of rank 3, thus degenerate. However, since it is not regular, it is not a valid counterexample to the Conjecture.

\begin{teo}\label{modregular}
If $M$ is a connected modular matroid of rank $r\geq 3$, then it is not regular.
\end{teo}
\begin{proof}
Firstly, we point out that a quick check shows that the statement is true for matroids on less than 8 elements (See for example \cite{list}). We also see that the only rank-3 connected modular matroids on 7 elements are the Fano matroid and its dual, which are known to be representable only on fields of characteristic 2 and, therefore, not regular. Consider now a rank-3 modular connected matroid $M$ on $n$ elements; since it is connected, in particular it does not contain any coloop, hence it is a connected projective geometry, where the hyperplanes coincide with the rank-2 flats. If $M$ contains any hyperplane $H$ with $k\geq 4$ elements then $M_H\cong U_{2,k}$, which is a minor of $M$ not representable in fields of characteristic 2. Thus, we just need to check what happens if all hyperplanes of $M$ have exactly 3 elements; by Theorem 7.2.5 in \cite{white} (See also \cite{stan2}), we can compute its characteristic polynomial to be equal to 
\begin{align*}
    \chi_M(t)& = (t-1)(t-2)(t-n+3)\\
    & = t^3-nt^2+(3n-7)t+(6-2n).
\end{align*}
However, its Tutte polynomial must be of the form
\begin{align*}
    T_M(x,y) =& \sum_{A\subseteq E}(x-1)^{3-\rk A}(y-1)^{|A|-\rk A}\\
    =& (x-1)^3+n(x-1)^2+{n\choose 2}(x-1) + n(x-1)(y-1) + \\
    & \left[{n\choose 3} -n \right]+\sum^n_{k=4} {n\choose k}(y-1)^{k-3},
\end{align*}
which means that the characteristic polynomial is equal to
\begin{align*}
    \chi_M(t) =& -T_M(1-t,0)\\
    =&t^3 - nt^2 +\frac{n(n-3)}{2}t - \frac{n^2-5n+2}{2}.
\end{align*}
The two expressions for $\chi_M$ are equal if and only if $n=7$, which means that $M$ is a Fano matroid (or its dual). This lets us conclude that the statement holds on rank-3 matroids. For a generic rank $r$, consider a connected projective geometry $M$ and all its rank-3 flats. If $G$ is a rank-3 flat, $M_G$ must be a modular (not necessarily connected) geometric lattice and thus must be isomorphic to one among these possibilities
\begin{itemize}
    \item $B_3$,
    \item $B_1\oplus U_{2,k}$,
    \item A rank-3 connected projective geometry.
\end{itemize}
However, $M$ is a connected projective geometry and its rank-2 flats need to have at least three elements, which lets us exclude the first two options. Therefore, every restriction is isomorphic to a rank-3 connected projective geometry, which is not regular.
\end{proof}

These results let us then conclude that 
\begin{itemize}
    \item If a matroid is modular, connected and regular it is isomorphic to either $B_1$ or $U_{2,3}$, which are both non-degenerate;
    \item If a connected regular matroid of rank $r\leq 4$ is not modular, then it is non-degenerate.
\end{itemize}

\end{document}